\documentclass{amsart}
\usepackage {amsmath, amscd}
\usepackage {amssymb}
\usepackage{amsthm}

\usepackage {bm}
\usepackage {indentfirst} %indent the first par after section
\numberwithin{equation}{section}

\def\<{\langle}
\def\>{\rangle}

\def\BB{{\mathcal B}}
\def\CC{{\mathcal C}}
\def\DD{{\mathcal D}}
\def\EE{{\mathcal E}}
\def\FF{{\mathcal F}}
\def\GG{{\mathcal G}}
\def\HH{{\mathcal H}}

\def\LL{{\mathcal L}}
\def\MM{{\mathcal M}}
\def\NN{{\mathcal N}}
\def\OO{{\mathcal O}}

\def\RR{{\mathcal R}}

\def\XX{{\mathcal X}}

\def\bbB{\mathbb{B}}

\def\bbC{\mathbb{C}}

\def\DDD{\mathfrak{D}}
\def\ddd{\mathfrak{d}}
\def\mmm{\mathfrak{m}}

\def\HHH{\mathfrak{H}}

\newcommand{\Span}{\mathop{\rm span}}

\newcommand{\HK}{\HHH(K)}

\newcommand{\bmB}{\bm{B}}
\newcommand{\tbB}{\tilde{\bm{B}}}

\newtheorem{lemma}{Lemma}[section]
\newtheorem{proposition}[lemma]{Proposition}%[section]
\newtheorem{theorem}[lemma]{Theorem}%[section]
\newtheorem{corollary}[lemma]{Corollary}

\theoremstyle{definition}
\newtheorem{remark}[lemma]{Remark}

\newtheorem{example}[lemma]{Example}

\title[Contractively included subspaces of Pick spaces]{Contractively included subspaces\\ of Pick spaces}
\author{Chafiq Benhida} 
\address{%
UFR de Math\'ematiques\\ Universit\'e des Sciences et Technologies
de Lille\\ F-59655 Villeneuve D'Ascq Cedex\\ France}
\email{Chafiq.Benhida@math.univ-lille1.fr}

\author{Dan Timotin}
\address{Institute of Mathematics of the Romanian Academy\\ P.O. Box 1-764\\ Bucharest 
014700\\ Romania}
\email{Dan.Timotin@imar.ro}

\thanks{The second author was partially supported by a grant of the Romanian National Authority for Scientific
Research, CNCS Ð UEFISCDI, project number PN-II-ID-PCE-2011-3-0119.}


\begin{document}

\begin{abstract}
Pick spaces are a class of reproducing kernel Hilbert spaces that generalize the classical Hardy space and the Drury--Arveson reproducing kernel spaces. We give characterizations of certain contractively included subspaces of Pick spaces. These generalize the characterization of closed invariant subspaces of Trent and McCullough, as well as results for the Drury--Arveson space obtained by Ball, Bolotnikov and Fang.
\end{abstract}

\maketitle

\section{Introduction}

An active area of research during the last decades has been the extension of function theory in classical spaces as the Hardy space and the Bergmann space  to other functional spaces. A first natural generalization is obtained by passing to several dimensions, that is, considering the domain of definition of the functions to be an open set in $\bbC^{n}$. This has lead to Hardy and Bergmann--type spaces in several variables (see, for instance,~\cite{rudin-polydisc-book1969,rudin-unitball-book1980}). 

Another point of view is to consider the functional spaces as reproducing kernel Hilbert spaces. 
An interesting space in this context is the so-called Drury--Arveson space~\cite{drury1983,arveson-subalgebrasiii-acta1998}. 
Viewed as a reproducing kernel space, the Drury--Arveson space is  the most important of a whole class of spaces whose reproducing kernel is characterized by a certain positivity condition. These \emph{Pick spaces} (alternately Nevanlinna--Pick or complete Pick)  have been introduced in~\cite{quiggin1993}; they contain several interesting spaces, as the Dirichlet--type spaces~\cite{shimorin-nevalinnapick-2002}, certain Sobolev spaces, etc. A basic reference that can be used is~\cite{agler-mccarthy-book}.

In particular, a Beurling type theorem for Pick spaces is proved in~\cite{mccullough-trent2000}, characterizing the closed subspaces of a Pick space that are invariant to multipliers. Actually,~\cite{mccullough-trent2000} discusses also vector-valued versions of the Pick spaces. As one should expect, a certain type of operator valued inner functions play the main role in the characterization.

Beurling's theorem and its generalizations refer to closed subspaces. A  natural sequel is the investigation of contractively included subspaces, which are no more necessarily close. Contractively included subspaces of Hilbert spaces have gained renewed interest
especially after the work of De Branges (see, for instance,~\cite{debrangesrovnyak-book1966}) referring to contractively included subspaces of the Hardy space.  A whole book of Sarason~\cite{sarason-subhardyspaces-book1994} has been later dedicated to this subject.

The  question of characterizing invariant contractively included subspaces of the Hardy space turns out to be more subtle. A good reference for this type of results is~\cite{nikolski-vasyunin1985}. The characterization has been extended to the Drury--Arveson space in~\cite{ball-bolotnikov-fang2007} (see also~\cite{ball-bolotnikov-fang2007-2,ball-bolotnikov-fang2008}), and  the passage to several dimensions brings to the surface some new phenomena. (It should  be noted  that in~\cite{ball-bolotnikov-fang2007} one discusses also the so--called noncommutative analogue of the Drury--Arveson space, the \emph{Fock space}, that has been much studied in papers of Popescu, starting with~\cite{popescugelu-indiana1989,popescugelu-jot1989, popescugelu-multianalyticfock-mathann1995}.)

The aim of this paper is to obtain analogues of the results of~\cite{ball-bolotnikov-fang2007} in the case of Pick spaces. Although~\cite{ball-bolotnikov-fang2007} is a source of inspiration, it should be noted that the general situation requires different arguments. One cannot anymore rely on the system theory interpretation which is the basis of much of the development in~\cite{ball-bolotnikov-fang2007}, nor can one use the universality property of the Drury--Arveson space in order to deduce the more general results. Also, one should note that the results in~\cite{ball-bolotnikov-fang2007,ball-bolotnikov-fang2007-2,ball-bolotnikov-fang2008} are given for a Drury--Arveson space with only a finite number of variables (although one might surmise that the argument can be extended to the infinite case). 
Moreover, some new phenomena appear in the general case that are worth mentioning. 

The plan of the paper is the following. We start with a section of preliminaries which introduces the required notions. The next section introduces our main characters, the Pick spaces. Section~\ref{se:invariant subspaces} deals with the analogue of Beurling's theorem; the main result is Theorem~\ref{th:invariant spaces}, which characterizes contractively included subspaces that are invariant to all multipliers. It is interesting to note that, contrary to the case of usual closed subspaces, this does not lead immediately to the characterization of their complementary subspaces (complementary subspaces are the generalization of orthogonal subspaces). Thus, in Section~\ref{se:complementary subspaces} we characterize these complementary subspaces in Theorem~\ref{th:characterization of complementary subspaces}. The final section discusses mostly some differences that appear between the Drury--Arveson space and the general Pick space.

\section{Preliminaries}\label{se:preliminaries}

\subsection{Contractively included subspaces}

Suppose $\HH, \HH'$ are Hilbert spaces, with norms denoted $\|\cdot\|$, $\|\cdot\|'$ respectively, and such that $\HH'\subset \HH$ as a vector subspace. If  the inclusion is a contraction; that is, $\|x\|_\HH\le \|x\|_{\HH'}$ for all $x\in \HH'$, we say that $\HH'$ is contractively included in $\HH$ and write $\HH'\Subset \HH$. 

It is often convenient to look at a contractively included subspace as the image of a contraction. More precisely, if $C:\CC\to \HH$, then we can define a complete Hilbert space norm on $C(\CC)$ by the formula
\begin{equation}\label{eq:definition of the range norm}
\|x\|_C:=\inf \{ \|y\|: Cy=x\}.
\end{equation}
We will denote this Hilbert space by $\RR_C$; we have $\RR_C\Subset \HH$. In case $C$ is injective, there is no need for the infimum, and $C$  maps $\CC$ unitarily onto $\RR_C$; otherwise, it maps $\ker C^\perp$ unitarily onto $\RR_C$. 

Of course, any $\HH'\Subset \HH$ is of the form $\RR_C$ if we take $C$ to be the injection of $\HH'$ into~$\HH$. 

A good reference for basic properties of contractively included subspaces is~\cite{nikolski-vasyunin1985}, Section 5. The next two lemmas gather the main facts that we will use; they are  consequences of Lemmas 5.7, 5.8 and 5.9 therein. We give the proof only for a point that is not explicit therein.

\begin{lemma}\label{le:same range space}
Suppose $C_1:\CC_1\to \HH$, $C_2:\CC_2\to \HH$ are contractions. Then $\RR_{C_1}=\RR_{C_2}$ (with equality of norms) if and only if $C_1C_1^*=C_2C_2^*$.
\end{lemma}

\begin{lemma}\label{le:contraction and commutation on range spaces}
Suppose $C_i:\CC_i\to \HH_i$ for $i=1,2$ are contractions, 
\begin{enumerate}
 \item If $D\in \LL(\CC_1, \CC_2)$ is a contraction, then there exists a unique contraction $T:\RR_{C_1}\to \RR_{C_2}$ such that $TC_1=C_2D$.

\item If $T\in\LL(\HH_1, \HH_2)$, then the following are equivalent:

(i) $T(\RR_{C_1})\subset \RR_{C_2}$, and $T:\RR_{C_1}\to \RR_{C_2}$ is a contraction.

(ii) There exists a contraction $D\in \LL(\CC_1, \CC_2)$ such that $TC_1=C_2D$.

(iii) $TC_1C_1^*T^*\le C_2C_2^*$.
\end{enumerate}
\end{lemma}

\begin{proof}
 We will prove only (1). Take $x\in \RR_{C_1}$. Suppose $y\in \CC_1$ and $C_1y=x$. We have then 
\[
 \|C_2Dy\|_{\RR_{C_2}}=\inf \{ \|z\|: z\in\CC_2,\, C_2z=C_2Dy \}\le \|Dy\| \le \|y\|.
\]
By taking the infimum with respect to all $y\in \CC_1$ such that $C_1y=x$, we obtain
\[
 \|C_2Dy\|_{\RR_{C_2}}\le \|x\|_{\RR_{C_1}}.
\]
Therefore, defining $Tx=C_2Dy$ yields a contraction that satisfies the $TC_1=C_2D$.
\end{proof}

A basic notion for contractively included subspaces is that of complementary subspace (see \cite{nikolski-vasyunin1985, sarason-subhardyspaces-book1994}). If $\HH'\Subset \HH$, then there exists a unique Hilbert space $\HH''\Subset \HH$ with the properties

(a) $\|x'+x''\|^2_\HH\le \|x'\|^2_{\HH'} + \|x''\|^2_{\HH''}$ for any $x'\in \HH'$, $x''\in \HH''$;

(b) for any $x\in \HH$ there exists a unique decomposition $x=x'+x''$, $x'\in \HH'$, $x''\in \HH''$, with $\|x\|^2_\HH= \|x'\|^2_{\HH'} + \|x''\|^2_{\HH''}$.

The space $\HH''$ is called the \emph{complementary subspace} of $\HH'$ and is denoted by $\HH'{}^\sharp$. The following lemma can be found in \cite{nikolski-vasyunin1985, sarason-subhardyspaces-book1994}.

\begin{lemma}\label{le:complementary space as image of defect}
If $\HH'=\RR_C$, then $\HH'{}^\sharp=\RR_{(I-CC^*)^{1/2}}$.
\end{lemma}

\subsection{Kernels and multipliers}

If $K$ is a positive definite kernel on a set $\Lambda$, we denote by $\HHH(K)$ the reproducing kernel space with kernel $K$. If we define $k_\lambda(\mu)=K(\mu, \lambda)$, then we have $\<f, k_\lambda\>= f(\lambda)$ for all $f\in \HHH(K)$.

Besides scalar valued reproducing kernels and corresponding reproducing kernel spaces, we will also have the opportunity to consider operator valued kernels; that is, taking values in $\LL(\EE)$ for some Hilbert space $\EE$. Such a kernel $K$ is called positive definite if for any choice of vectors $\xi_i\in\EE$ we have $\sum_{i, j} \<K(\lambda_i, \lambda_j)\xi_i, \xi_j\>\ge 0$.

In particular, we may consider the tensor Hilbert space product $\HK\otimes\GG$ for some Hilbert space $\GG$. If $F\in \HK\otimes \GG$, then we can define the value $F(\lambda)\in \GG$ by the formula
\[
\<F(\lambda) ,\xi \>= \< F,k_\lambda\otimes\xi \>.
\]
It is easily seen that this formula gives for simple tensors $f\otimes \xi$ the value $f(\lambda)\xi$. So elements in $\HK\otimes \GG$ can be viewed as functions on $\Lambda$ with values in $\GG$.

A simple example of positive definite operator valued kernel on $\EE$ can be obtained as follows. Take a function  $G:\Lambda\to\LL(\GG', \EE)$, where $\GG'$ is an arbitrary Hilbert space. The kernel 
\begin{equation}\label{eq:projecting the kernel}
 L_G(\mu, \lambda)=G(\mu)G(\lambda)^*
\end{equation}
 is positive definite.

Suppose one is given now a function $G:\Lambda\to\LL(\GG', \GG)$ for some Hilbert spaces $\GG, \GG'$. We are interested when such functions generate contractive multipliers from $\GG'\otimes\HK$ to $\GG\otimes\HK$. A total family in the latter space is given by the set $\xi\otimes k_\lambda$, with $\xi\in \GG$ and $\lambda\in \Lambda$. We say that $G$ is a \emph{multiplier}  if the operator densely defined by $\mmm_G^*(\xi\otimes  k_\lambda) = G(\lambda)^* \xi\otimes k_\lambda$ can be extended to a bounded operator; $G$ is a \emph{contractive multiplier} if it can be extended to a contraction. (Note that we have actually defined  $\mmm_G$ through its adjoint.)

If $G$ is a scalar multiplier (that is, $\GG=\GG'=\bbC$), then we will write $M_G$ (acting on $\HK$) instead of $\mmm_G$.

\begin{lemma}\label{le:kernel condition for contractive multipliers}
 With the above notations, $G$ is a contractive multiplier iff the operator valued kernel
\[
 (I_{\GG}- G(\mu)G(\lambda)^*)K(\mu, \lambda)
\]
is positive definite.
\end{lemma}

\begin{proof}
 The proof is obtained by writing the action of $\mmm_G^*$ on a linear combination from the family of total vectors on which it is originally defined.
\end{proof}

We will have the opportunity to consider triple tensor product spaces with one of the factors being $\HK$. It will be convenient in this case to write $\HK$ in the middle.

Suppose that we have two multipliers $G:\Lambda\to \LL(\GG',\GG)$, $F:\Lambda\to \LL(\FF',\FF)$. 
A total set in the tensor product $\GG\otimes \HK\otimes \FF$ is formed by elements of the form $\xi\otimes k_\lambda\otimes \eta$ ($\xi\in \GG, \eta\in \FF$). With a slight abuse of notation, we may write
\begin{equation*}
\begin{split}
(I_{\GG'}\otimes \mmm_F^*)(\mmm_G^*\otimes I_\FF)(\xi\otimes k_\lambda\otimes \eta)&= G(\lambda)^*\xi\otimes k_\lambda\otimes F(\lambda)^*\eta\\
&= (\mmm_G^*\otimes I_{\FF'})(I_\GG\otimes \mmm_F^*)(\xi\otimes k_\lambda\otimes \eta),
\end{split}
\end{equation*}
and therefore
\[
(I_\GG\otimes \mmm_F)(\mmm_G\otimes I_{\FF'})=(\mmm_G\otimes I_\FF)(I_{\GG'}\otimes \mmm_F).
\]
Define $\MM_G$ to be the image of $\mmm_G$, then $\MM_G$ is a vector subspace of $\GG\otimes\HK$, not necessarily closed. The above formula shows that for any $F:\Lambda\to \LL(\FF',\FF)$, we have
\[
(I_\GG\otimes \mmm_F)(\MM_G\otimes \FF')\subset (\MM_G\otimes \FF).
\]
We will shorten this property by saying that $\MM_G$ is \emph{completely invariant} to multipliers.

Suppose now that $G$ is a contractive multiplier. Then $\MM_G=\RR_{\mmm_G}$ as vector spaces, and we may define a norm on $\MM_G$ by applying formula~\eqref{eq:definition of the range norm}. Also, for any other Hilbert space $\FF$ we may define a norm on $\MM_G\otimes \FF$ by the same formula applied to $\mmm_G\otimes I_\FF$.

\begin{lemma}\label{le:the easy part of the subspace representation}
With the above notations, suppose $G$ is a contractive multiplier. Then  $\MM_G$ is a contractively included subspace of $\GG\otimes\HK$, completely invariant to multipliers. Moreover, if $F:\Lambda\to \LL(\FF',\FF) $ is a contractive multiplier, then $I_\GG\otimes \mmm_F$ acts contractively from $\MM_G\otimes \FF'$ to $\MM_G\otimes \FF$.
\end{lemma}

\begin{proof}
The first part of the lemma follows from the preceding comments. As for the last part, we use Lemma~\ref{le:contraction and commutation on range spaces}. If we take
\[
\begin{split}
&\CC_1=\GG'\otimes \HK\otimes \FF',\quad \CC_2=\GG'\otimes \HK\otimes \FF,\\
&\HH_1=  \GG\otimes \HK\otimes \FF',\quad \HH_2=  \GG\otimes \HK\otimes \FF,\\
&C_1=\mmm_G\otimes I_{\FF'},\quad C_2=\mmm_G\otimes I_{\FF},\\ &D=I_{\GG'}\otimes\mmm_F, \quad T=I_{\GG}\otimes\mmm_F,
\end{split}
\]
then condition (ii) therein is satisfied; then (i) gives the desired result.
\end{proof}

The following result from \cite{aronszajn1950} identifies the complementary subspace of $\MM_G$. 

\begin{lemma}\label{le:identification of the complementary subspace by its kernel}
The space $\MM_G^\sharp\Subset \GG\otimes \HK$ is the reproducing kernel space with kernel $(I_{\GG}- G(\mu)G(\lambda)^*)K(\mu, \lambda)$ (which is positive definite by Lemma~\ref{le:kernel condition for contractive multipliers}).
\end{lemma}

\section{Pick spaces}\label{se:pick spaces}

A \emph{Pick kernel} $K$ defined on a space $\Lambda$ is characterized by the property that for any $\lambda_0\in\Lambda$ the function 
\[
 1-\frac{K(\mu, \lambda_0) K(\lambda_0, \lambda)}{K(\mu, \lambda) K(\lambda_0, \lambda_0)}
\]
is also a positive definite kernel. The corresponding reproducing kernel space $\HK$ is called a \emph{Pick space}. The point $\lambda_0$ is called the \emph{base  point}. Its choice is in fact not important; if we assume, for instance, that $K(\mu, \lambda)\not=0$ everywhere, then the condition is independent of the choice of the base point (see~\cite{agler-mccarthy-book}). Note that in~\cite{agler-mccarthy-book} this is called the \emph{complete Pick} property.

If we denote by $\delta$ the normalization of $k_{\lambda_0}$ (that is,    $ \delta=\frac{k_{\lambda_0}}{\|k_{\lambda_0}\|}$),
one can rewrite the above condition as saying that there exists a Hilbert space $\BB$ and a function $\beta:\Lambda\to \BB$ such that
\begin{equation}\label{eq:pick}
  K(\mu, \lambda)=\frac{\delta(\mu)\overline{\delta(\lambda)}}{1-\<\beta(\lambda), \beta(\mu)\>}.
\end{equation}
In particular, by taking $\lambda=\mu$, we obtain that $\|\beta (\lambda)\|<1$ for all $\lambda$. 

One should note here that $\BB$ and $\beta$ are essentially uniquely defined by the minimality condition $\BB=\Span\{\beta(\lambda):\lambda\in\Lambda\}$. Indeed, if $\beta':\Lambda\to\BB'$ also satisfies~\eqref{eq:pick}, and the corresponding minimality condition, then the map $\beta(\lambda)\mapsto \beta'(\lambda)$ extends to a unitary $U:\BB\to\BB'$ such that $\beta'=U\circ\beta$. We will always assume that the minimality condition is satisfied.

There is a certain contractive multiplier associated to a Pick kernel that will play an important role in the sequel. Namely, define $B(\lambda):\BB\to \bbC$ by the formula
\begin{equation}\label{eq:definition of B}
B(\lambda)\xi=\<\xi, \beta (\lambda) \>.
\end{equation}
We may rewrite~\eqref{eq:pick} as
\[
 (1-B (\mu)B(\lambda)^*)K(\mu, \lambda)=\delta(\mu)\overline{\delta(\lambda)}.
\]
The right hand side term is positive definite, and then so is the left hand side term. Taking into account Lemma~\ref{le:kernel condition for contractive multipliers}, this means that $B(\lambda)$ is a contractive multiplier from $\BB$ into $\bbC$. We denote $\bmB=\mmm_B:\BB\otimes \HK\to \HK$. One checks easily from the definition that $B(\lambda)^*1=\beta(\lambda)$ and 
\begin{equation}\label{eq:B*(k_lambda)}
 \bmB^*(k_\lambda)=\beta(\lambda)\otimes k_\lambda.
\end{equation}

There is a slight difference between our convention and the one that usually appears in the theory of Pick kernels, as for instance in \cite[Ch. 8]{agler-mccarthy-book}. Namely, in formula~\eqref{eq:pick} the denominator in the right hand side is $1-\< \beta(\mu), \beta(\lambda)\>$. Since we are mainly interested in the corresponding multiplier $B$, it is more natural to introduce the function $\beta$ as in~\eqref{eq:pick}. But it will also be useful to consider an arbitrary conjugation $J$ acting on the Hilbert space $\BB$, and to define $\bar\beta(\lambda)=J\beta(\lambda)$. The function $\bar\beta$ is related to an embedding theorem for Pick spaces, that we discuss in the sequel.

Suppose $\bbB$ is the unit ball in the Hilbert space $\BB$. By $\DD(\bbB)$ we will denote the \emph{Drury--Arveson} space~\cite{drury1978, arveson-subalgebrasiii-acta1998}, which is the reproducing kernel space corresponding to $\Lambda=\bbB$ and reproducing kernel
\[
\DDD (\eta, \xi)=\frac{1}{1-\<\eta, \xi\>}.
\]
This is called the ``universal Pick space''  in~\cite{agler-mccarthy-book}. For a point $\xi\in \bbB$, $\ddd_\xi$ is the corresponding reproducing kernel in $\DD(\bbB)$, and $\psi_\xi$ is the function $\psi_\xi(\eta):= \< \eta,\xi \>$. Then all $\psi_\xi$s  are contractive multipliers, and a common eigenvector of these contractive multipliers has to be a reproducing kernel.

Based on the obvious relation $K(\mu, \lambda)=\DDD(\bar\beta (\mu), \bar\beta(\lambda))$, in~\cite[Theorem 8.2]{agler-mccarthy-book},  the following embedding result is proved.

\begin{theorem}\label{th:embedding}
The map $k_\lambda\mapsto \ddd_{\bar\beta(\lambda)}$ extends to an isometric linear embedding $\epsilon_K$ from $\HK$ into $\DD(\bbB)$. The adjoint $\epsilon_K^*$ of this embedding is composition with $\bar\beta$.
\end{theorem}

It is easy to see that  $\epsilon_K$  is unitary if and only if the image of $\bar\beta$ is a uniqueness set for $\DD(\bbB)$, that is, if $f\in \DD(\bbB)$ and $f|E\equiv 0$ imply $f\equiv0$.

\begin{lemma}\label{le:M*phi J=J M*phi o b}
If $\phi$ is a contractive multiplier on $\DD(\bbB)$, then $\phi\circ \bar\beta$ is a contractive multiplier on $\HK$, and we have
\[
M^*_\phi   \epsilon_K=  \epsilon_K M^*_{\phi\circ \bar\beta},
\]
where the multiplier on the left acts in $\DD(\bbB)$ and the one on the right in $\HK$.
\end{lemma}

\begin{proof}
The first part of the lemma follows from Lemma~\ref{le:kernel condition for contractive multipliers}. As for the equality, it has to be checked on reproducing kernels, where we have
\[
M_\phi^*( \epsilon_Kk_\lambda)= M_\phi^*(\ddd_{\bar\beta(\lambda)})=
\overline{\phi(\bar\beta(\lambda))} \ddd_{\bar\beta(\lambda)}=
 \epsilon_K(\overline{\phi(\bar\beta(\lambda))} k_\lambda= M^*_{\phi\circ \bar\beta}k_\lambda.\qedhere
\]
\end{proof}

Let us finally note that for any $\xi\in\bbB$ the function $p_\xi(\eta):= \< \eta,\xi \>$ is a contractive multiplier, and a common eigenvector of these contractive multipliers has to be a reproducing kernel.

\section{Completely invariant contractively included subspaces of Pick spaces}\label{se:invariant subspaces}

\begin{lemma}\label{le:delta}
We have $\delta\in \HHH(K)$ and $\|\delta\|=1$, and $I-\bmB\bmB^*$ is  the projection onto the space generated by $\delta$.
\end{lemma}

\begin{proof}
The first statement follows immediately from the definition of $\delta$. For the second, take two reproducing kernels $k_\mu, k_\lambda$. Then
\[
\begin{split}
\< (I-\bmB\bmB^*)k_\lambda,k_\mu \>&=
\< k_\lambda,k_\mu \>- \< \bmB^*k_\lambda,\bmB^*k_\mu \>\\
&= K(\mu, \lambda)- \< B(\lambda)^*k_\lambda, B(\mu)^*k_\mu \>\\
&= (1-B (\mu)B(\lambda)^*)K(\mu, \lambda)=\delta(\mu)\overline{\delta(\lambda)}.
\end{split}
\]
On the other hand, if we define $\pi f= \< f,\delta \>\delta$, then 
\[
\< \pi k_\lambda,k_\mu \>= \< k_\lambda, \delta\> \cdot \< \delta, k_\mu \> = \delta(\mu)\overline{\delta(\lambda)}.
\]
Thus $I-\bmB\bmB^*=\pi$, which proves the lemma.
\end{proof}

The main result below concerns contractively included subspaces. The proof is suggested by that  of~\cite[Theorem 0.7]{mccullough-trent2000}.

\begin{theorem}\label{th:invariant spaces}
 Suppose $K$ is a Pick kernel, $\GG$ is a Hilbert space, and $\MM$ is  a contractively included subspace of $\HK\otimes\GG$. Then the following are equivalent:
 
(i) $\MM$ is completely invariant to multipliers, and, moreover, if $F:\Lambda\to \LL(\FF', \FF)$ is a contractive multiplier, then $ \mmm_F\otimes I_\GG$ acts contractively from $\FF'\otimes \MM$ to $\FF\otimes \MM$.

(ii) $(\bmB\otimes I_\GG)(\BB\otimes\MM)\subset \MM$, and $\bmB\otimes I_\GG$ acts contractively on from $\BB\otimes\MM$ to $\MM$.

 (iii) There exists a Hilbert space $\GG'$ and a contractive multiplier $G:\Lambda\to\LL(\GG', \GG)$ such that $\MM=\MM_G$.
\end{theorem}

\begin{proof}
(iii)$\Rightarrow$(i). This has been proved in Lemma~\ref{le:the easy part of the subspace representation}.

\smallskip

(i)$\Rightarrow$(ii). Since $\bmB$ is a contractive multiplier, (ii) follows immediately from (i).

\smallskip
(ii)$\Rightarrow$(iii). Let us denote by $C:\MM\to \GG\otimes \HK$ the inclusion (which is known to be a contraction). 
Applying Lemma~\ref{le:contraction and commutation on range spaces} to the case $C_1=I_\BB\otimes C$, $C_2=C$, and $T=\bmB\otimes I_\GG$, we obtain
$( \bmB\otimes I_\GG) (I_\BB\otimes CC^*) (\bmB^*\otimes I_\GG)\le CC^*$, or 
\[
 CC^*-( \bmB\otimes I_\GG) (I_\BB\otimes CC^*) (\bmB^*\otimes I_\GG)\ge0,
\]
so we may write 
\begin{equation}\label{eq:definition of X}
  CC^*-( \bmB\otimes I_\GG) (I_\BB\otimes CC^*) (\bmB^*\otimes I_\GG)=XX^*
\end{equation}
for some Hilbert space $\GG'$ and operator $X:\GG'\to\HK\otimes \GG$.
 
On the other hand, for $\lambda\in\Lambda$, $\xi\in\GG$ we have
\[
 (I_\BB\otimes C^*)(\bmB^*\otimes I_\GG) ( k_\lambda\otimes\xi)=
(I_\BB\otimes C^*)(B(\lambda)^*\otimes k_\lambda\otimes\xi)=
B(\lambda)^*\otimes C^* (k_\lambda\otimes\xi)
\]
and so, if $\lambda, \mu\in \Lambda$, $\xi, \eta\in \GG$, then
\begin{equation}\label{eq:first sum b_j}
\begin{split}&
\< CC^*-( \bmB\otimes I_\GG) (I_\BB\otimes CC^*) (\bmB^*\otimes I_\GG) (k_\lambda\otimes\xi ),(k_\mu\otimes\eta )\>\\&\qquad =
\< C^*(k_\lambda\otimes\xi), C^*(k_\mu\otimes\eta)  \>\\
&\qquad\qquad\qquad-
 \<( I_\BB\otimes C^*)  (\bmB^*\otimes I_\GG) ( k_\lambda\otimes\xi) , ( I_\BB\otimes C^*)  (\bmB^*\otimes I_\GG) ( k_\mu\otimes\eta) \>\\
&\qquad =(1-B(\mu) B(\lambda)^*) \< C^*(k_\lambda\otimes\xi), C^*(k_\mu\otimes\eta  ) \>\\
&\qquad= \frac{\overline{\delta(\lambda)}\delta(\mu)}{K(\mu, \lambda)}
  .\< C^*(k_\lambda\otimes\xi), C^*(k_\mu\otimes\eta)  \>.
\end{split}
\end{equation}

Define then the function $G:\Lambda\to \LL(\GG',\GG)$ by the formula
\[
G(\lambda)^* \xi=
\begin{cases}
\frac{1}{\overline{\delta(\lambda)}} X^*( k_\lambda\otimes\xi)& \text{ if }\delta(\lambda)\not=0,\\
0&\text{ otherwise}.
\end{cases}
\]
Then, if $\delta(\lambda)\not=0$,
\[
\mmm_G^*(k_\lambda\otimes\xi)=\frac{1}{\overline{\delta(\lambda)}} [X^*(k_\lambda\otimes\xi)]\otimes k_\lambda,
\]
and thus, if $\delta(\lambda)\not=0$, $\delta(\mu)\not=0$, then, using~\eqref{eq:first sum b_j} and~\eqref{eq:definition of X}, we obtain
\[
\begin{split}
&\<\mmm_G \mmm_G^* (k_\lambda\otimes\xi), (\eta\otimes k_\mu)  \>\\
&\qquad= \frac{1}{\overline{\delta(\lambda)}\delta(\mu)}\< [X^*(k_\lambda\otimes\xi)]\otimes k_\lambda, [X^*(k_\mu\otimes\eta)]\otimes k_\mu\>\\
&\qquad =\frac{K(\mu, \lambda)}{\overline{\delta(\lambda)}\delta(\mu)} \< XX^*(k_\lambda\otimes\xi),(k_\mu\otimes\eta)\>\\
&\qquad= \frac{K(\mu, \lambda)}{\overline{\delta(\lambda)}\delta(\mu)}\<( CC^*-( \bmB\otimes I_\GG) (I_\BB\otimes CC^*) (\bmB^*\otimes I_\GG)) (k_\lambda\otimes\xi),(k_\mu\otimes\eta)\>\\
&\qquad=\< C^*(k_\lambda\otimes\xi), C^*(k_\mu\otimes\eta)  \>=\< CC^*(k_\lambda\otimes\xi), (k_\mu\otimes\eta)  \>.
\end{split}
\]
Therefore $\mmm_G \mmm_G^*=CC^*$, whence it follows, by Lemma~\ref{le:same range space}, that $\MM=\MM_G$.
\end{proof}

The next corollary is a reformulation of the main part of~\cite[Theorem 0.7]{mccullough-trent2000}. 

\begin{corollary}\label{co:th0.7 trent-mccullough}
Suppose that $\MM\subset  \HK\otimes\GG$ is a closed subspace with the property that $(\bmB\otimes I_\GG)(\BB\otimes\MM)\subset\MM$. Then there exists $\GG'$ and a contractive multiplier $G:\Lambda\to\LL(\GG', \GG)$ such that $\mmm_G\mmm_G^*$ is the projection onto $\MM$. In particular, $\MM=\MM_G$.
\end{corollary}

\begin{proof}
Since $\bmB$ is contractive, $\bmB\otimes I_\GG:\BB\otimes\HK\otimes \GG\to\HK\otimes\GG$ is also contractive. As $\MM$ is a closed subspace of $\HK\otimes \GG$,  it follows that  $\MM$ satisfies condition (ii) of Theorem~\ref{th:invariant spaces}. 
Thus Theorem~\ref{th:invariant spaces} provides us with the contractive multiplier $G$. By Lemma~\ref{le:same range space}, we must have $\mmm_G\mmm_G^*$ equal to the orthogonal projection onto $\MM$.
\end{proof}

Contractive multipliers $G$ which have the property that $\mmm_G$ is a partial isometry are called in~\cite{mccullough-trent2000} \emph{inner} multipliers.

One can also obtain an analogue of~\cite[Theorem 0.14]{mccullough-trent2000}. Since the proof is very similar,  we just give the corresponding statement.

\begin{theorem}\label{th:two subspaces contractively included}
Suppose $\MM_{G_1}\Subset \MM_{G_2}$. Then there exists a contractive multiplier $\Gamma:\Lambda\to \LL(\GG'_1, \GG'_2)$, such that $G_1=G_2\Gamma$.  
\end{theorem}

\begin{remark}\label{re:relation to other}
To relate the above results to the statements in~\cite{mccullough-trent2000}, note that our function $\beta$ is written therein in coordinates; more precisely, if one chooses an orthonormal basis $(e_i)$ in $\BB$, then the functions $b_i$ that appear in~\cite{mccullough-trent2000} are related to our $\beta$ by the equality $b_i(\lambda)= \< e_i ,\beta(\lambda)\>$. The relation to~\cite{ball-bolotnikov-fang2007} is even simpler, since in that case the basis $(e_i)$ is already explicit in the definition of the reproducing kernel. 
\end{remark}

\section{Complementary subspaces in Pick spaces}\label{se:complementary subspaces}

\subsection{Complementary subspaces of ranges of multipliers}

Once we have identified the completely invariant subspaces of Pick spaces, we may go further and discuss their complementary subspaces. For this we need to know more about the structure of contractive multipliers. We pick the results we need from~\cite{ambrozie-timotin-ieot2002}.

For a Hilbert space $\XX$ and $\lambda\in \Lambda$,  define $Z_\XX(\lambda):\BB\otimes \XX\to \XX$ by $Z_\XX(\lambda)=B(\lambda)\otimes I_\XX$.
%
%its adjoint 
%\[
%Z_\XX(\lambda)^* x=\beta(\lambda)\otimes x.
%\]
%Note that, in particular, according to~\eqref{eq:B*(k_lambda)}, we have $B(\lambda)=Z_\bbC(\lambda)$ and $Z_\XX(\lambda)=B(\lambda)\otimes I_\XX$.
%
Suppose that $G:\Lambda\to\LL(\GG', \GG)$ is a contractive multiplier. Then there exists a Hilbert space $\XX$ and a coisometry $U:\XX\oplus \GG'\to (\BB\otimes\XX)\oplus \GG$ whose matrix with respect to the above decomposition is
\[
U=\begin{pmatrix}
\bm a & \bm b\\ \bm c & \bm d
\end{pmatrix}
\]
such that 
\begin{equation}\label{eq:representation of G}
G(\lambda)=\bm d+ \bm c (I-Z_\XX(\lambda)\bm a)^{-1} Z_\XX(\lambda)\bm b.
\end{equation}
Denote also 
\[
\OO(\lambda)=(I-\bm a^*Z_\XX(\lambda)^*)^{-1}\bm c^*:\GG\to \XX.
\]

%With the above notations we intend to identify $\MM_G^\sharp$. 

\begin{lemma}\label{le:formula for psi/gamma}
Define $\gamma:\XX\to \HK\otimes \GG$ by the formula $\gamma^*(k_\lambda\otimes y)=\overline{\delta(\lambda)}\OO(\lambda) y$. Then
\[
\mmm_G\mmm_G^*+\gamma\gamma^*=I_{\HK\otimes \GG}.
\]
\end{lemma}

\begin{proof}
From~\eqref{eq:representation of G} it follows that
\[
G(\lambda)^*=\bm d^*+\bm b^*Z_\XX(\lambda)^* \OO(\lambda),
\]
while the definition of $\OO$ yields
\[
\OO(\lambda)=\bm c^* + \bm a^* Z_\XX(\lambda)^* \OO(\lambda).
\]
Since $U^*$ is an isometry, the last two relations say (using also the definition of $Z_\XX(\lambda)^*$) that, for $\lambda,\mu\in \Lambda$ and $y, z\in \GG$,
\[
\left\langle
\begin{pmatrix}
\OO(\lambda)y\\ G(\lambda)^* y
\end{pmatrix},
\begin{pmatrix}
\OO(\mu)z\\ G(\mu)^* z
\end{pmatrix}
\right\rangle=
\left\langle
\begin{pmatrix}
\beta(\lambda)\otimes \OO(\lambda)y \\
y
\end{pmatrix},
\begin{pmatrix}
\beta(\mu)\otimes \OO(\mu) z \\ z
\end{pmatrix}
\right\rangle.
\]
This is equivalent to
\[
\<\OO(\lambda)y, \OO(\mu)z\>+\<G(\lambda)^* y,  G(\mu)^* z\>=
\<\beta(\lambda), \beta(\mu)\> \<\OO(\lambda)y, \OO(\mu)z\>+\<y, z\>,
\]
or, using the definitions of $\gamma$ and $\beta$, 
\[
\<\gamma^*(k_\lambda\otimes y), \gamma^*(k_\mu\otimes z)\>
+\<\mmm_G^*(k_\lambda\otimes y), \mmm_G^*(k_\mu\otimes z)\>=
\<k_\lambda\otimes y, k_\mu\otimes z\>.
\]
The proof is finished.
\end{proof}

\begin{corollary}\label{co:M_Gsharp}
With the above notations, $\MM_G^\sharp=\RR_\gamma$.
\end{corollary}

\begin{proof}
Since $\MM_G=\RR_{\mmm_G}$, it follows from Lemma~\ref{le:complementary space as image of defect} that $\MM_G^\sharp=\RR_{(I-\mmm_G\mmm_G^*)^{1/2}}$. But Lemma~\ref{le:formula for psi/gamma} says that 
\[
\gamma\gamma^*=I-\mmm_G\mmm_G^*= (I-\mmm_G\mmm_G^*)^{1/2}(I-\mmm_G\mmm_G^*)^{1/2},
\]
whence $\RR_\gamma=\RR_{(I-\mmm_G\mmm_G^*)^{1/2}}$ by Lemma~\ref{le:same range space}.
\end{proof}

For further use, let us note that we have
\[
\<\gamma(x)(\lambda), y\>= \< \gamma(x),k_\lambda\otimes y \>
= \< x,\overline{\delta(\lambda)}\OO(\lambda) y \> ,
\]
and thus
\begin{equation}\label{eq:formula for gamma}
\gamma(x)=\delta(\lambda)\OO(\lambda)^*x=
\delta(\lambda) \bm c (I-Z_\XX(\lambda)\bm a)^{-1}x.
\end{equation}

\subsection{Characterization of the subspaces: a special case}\label{se:delta=1} In the sequel we make a more restrictive assumption. Namely, we assume that the reproducing kernel is ``normalized at $\lambda_0$'', which means (see~\cite[Chapter 2.6]{agler-mccarthy-book}) that $\delta\equiv 1$. If we take then $\mu=\lambda=\lambda_0$ in formula~\eqref{eq:pick}, we obtain
\[
 1=\frac{1}{1-\<\beta(\lambda_0),\beta(\lambda_0)\>},
\]
whence $\beta(\lambda_0)=0$. Also, $\pi$ is in this case the orthogonal projection onto the constant function~$\bm 1$.

It follows from Corollary~\ref{co:M_Gsharp} that we also have
\[
 (\BB\otimes\MM_G)^\sharp= \RR_{I_\BB\otimes \gamma}.
\]
We apply then Lemma~\ref{le:contraction and commutation on range spaces} (1) to the case $\CC_1=\XX$, $\CC_2=\BB\otimes\XX$, $\HH_1=\HK\otimes\GG$, $\HH_2=\BB\otimes \HK\otimes\GG$, and $D=\bm a$. As a result, we obtain that there exists a unique contraction 
\[
 \tbB: \MM_G^\sharp\to (\BB\otimes\MM_G)^\sharp
\]
that verifies
\begin{equation}\label{eq:definition of tilde B}
\tbB \gamma= (I_\BB\otimes\gamma)\bm a.
\end{equation}

The reason to introduce the operator $\tbB$ is the following. Theorem~\ref{th:invariant spaces} characterizes the contractively included subspaces of $\HK\otimes G$ that satisfy $(\bmB\otimes I_\GG)(\BB\otimes \MM)\subset \MM$ as range spaces of multipliers. It seems tempting to search  a similar characterization for the complementary subspaces of range spaces of multipliers, and the natural candidate would be a contractively included subspace $\NN$ that satisfies 
\begin{equation}\label{eq:inclusion with B}
(\bmB^*\otimes I_\GG)\NN\subset \BB\otimes \NN.
\end{equation}
However, in~\cite{ball-bolotnikov-fang2007} there an example is given (in the Drury--Arveson space $\DD(\bbB)$) of a multiplier such that this inclusion is not true for the complementary subspace of its range.  In this context $\tbB$ is a ``replacement'' for $\bmB^*\otimes I_\GG$: we have, by construction, $\tbB(\MM_G^\sharp)\subset (\BB\otimes\MM_G)^\sharp$, and we show in the next Lemma that it also  satisfies some special relations, which will provide the basis for the desired characterization.

\begin{lemma}\label{le:difference equality}
 Suppose $\xi\in \RR_\gamma$. Then
\begin{align}
\label{eq:gleason}
 \xi&=(\bmB\otimes I_\GG)\tbB \xi+ (\pi\otimes I_\GG) \xi;\\
 \label{eq:inequality for difference quotients}
 \|\tbB \xi\|^2_{\BB\otimes\RR_\gamma}&\le \|\xi\|^2_{\RR_\gamma}
-\|(\pi\otimes I_\GG)\xi\|^2_{\HK\otimes\GG}.
\end{align}
\end{lemma}

In the classical case, \eqref{eq:inequality for difference quotients} is called the \emph{inequality for difference quotients}.

\begin{proof}
 Suppose $\xi=\gamma x$, with $x\in \XX$. Then
\[
 \xi-(\bmB\otimes I_\GG)\tbB \xi= \gamma x- (\bmB\otimes I_\GG)\tbB\gamma x=
\big(\gamma - (\bmB\otimes I_\GG)(I_\BB\otimes\gamma)\bm a\big)x,
\]
and $(\pi\otimes I_\GG)\xi=(\pi\otimes I_\GG)\gamma x$.  
The operators $D:=\gamma - (\bmB\otimes I_\GG)(I_\BB\otimes\gamma)\bm a $
and $D':= (\pi\otimes I_\GG)\gamma$ both act from $\XX$ to $\HK\otimes \GG$; we will show that they are equal by computing the action of their adjoints on an
 element of the form $k_\lambda\otimes y$. 

According to Lemma~\ref{le:formula for psi/gamma}, we have  $\gamma^*(k_\lambda\otimes y)=\OO(\lambda )y$, and thus, using the definitions of $Z_\XX(\lambda)$ and of $\OO(\lambda)$,
\[
\begin{split}
 D^* (k_\lambda\otimes y)
&=\OO(\lambda )y- \bm a^* (I_\BB\otimes \gamma^*)(\beta(\lambda)\otimes k_\lambda\otimes y)
=\OO(\lambda )y - \bm a^*(\beta(\lambda)\otimes \OO(\lambda )y)\\
&=\OO(\lambda)y- \bm a^*(Z_\XX(\lambda)\OO(\lambda)y=\bm c^* y.
\end{split}
\]

On the other hand,  using the fact that $\bm 1=k_{\lambda_0}$,
\[
D'{}^*(k_\lambda\otimes y)=
 \gamma^* (\pi\otimes I_\GG)(k_\lambda\otimes y)=
\gamma^* ( \<k_\lambda,\bm 1\>\bm 1\otimes y)=
\gamma^*(\bm 1\otimes y)=\OO(\lambda_0)y.
\]
But $\beta(\lambda_0)=0$ implies $Z_\XX(\lambda_0)=0$, and thus $\OO(\lambda_0)=\bm c^*$. 
Thus $D^*=D'{}^*$, $D=D'$, proves~\eqref{eq:gleason}.

To prove~\eqref{eq:inequality for difference quotients}, note that we
have obtained $\gamma^* (\pi\otimes I_\GG)(k_\lambda\otimes y)=\bm c^* y$ for all $\lambda\in\Lambda$ and $y\in\GG$. Taking the scalar product with $z\in \XX$, we have
\[
 \<\gamma^* (\pi\otimes I_\GG)(k_\lambda\otimes y), z\>=\<y, \bm c z\>
=\<k_\lambda\otimes y, \bm 1\otimes \bm c z\>
\]
and thus
\begin{equation}\label{eq:formula for pi otimes I gamma}
 (\pi\otimes I_\GG)\gamma z=\bm 1\otimes\bm c z.
\end{equation}

If $\xi\in\RR_\gamma$, $\xi=\gamma x$ for some $x\in \XX$, we have
\[
 \|\tbB \gamma x\|^2_{\BB\otimes\RR_\gamma} =
\|(I_\BB\otimes\gamma) \bm a x\|^2_{\BB\otimes\RR_\gamma}\le \|\bm ax\|_{\BB\otimes\XX}^2,
\]
the inequality following from the definition of the range norm. But the contractivity of $U$ and~\eqref{eq:formula for pi otimes I gamma} imply that 
\[
 \|\bm ax\|_{\BB\otimes\XX}^2= \|x\|^2_{\XX}-\|\bm cx\|_\GG^2
= \|x\|^2_{\XX}- \|(\pi\otimes I_\GG)\gamma x\|^2_{\HK\otimes\GG}.
\]
Taking the infimum with respect to all $x\in\XX$ such that $\xi=\gamma x$, we obtain~\eqref{eq:inequality for difference quotients}.
\end{proof}

Let us note that from Lemma~\ref{le:delta} it follows that for any $\xi\in\HK\otimes\GG$ we have
\begin{align}
\xi&=(\bmB \bmB^*\otimes I_\GG)\xi+(\pi\otimes I_\GG)\xi,
\label{eq:gleason for the whole space}\\
 \|(\bmB^*\otimes I_\GG)\xi\|_{\HK\otimes\GG}^2&=\|\xi\|_{\HK\otimes\GG}^2-\|(\pi\otimes I_\GG)\xi\|^2_{\HK\otimes\GG}.
\label{eq:diff quotients for the whole space}
\end{align}
These equalities should be compared to~\eqref{eq:gleason} and~\eqref{eq:inequality for difference quotients}.

Lemma~\ref{le:difference equality} is the basis for  a characterization of contractively contained subspaces of $\HK\otimes\GG$ that are complementary spaces of contractively included completely invariant spaces.

\begin{theorem}\label{th:characterization of complementary subspaces}
 Suppose $K$ is a Pick kernel, $\GG$ is a Hilbert space, and $\NN$ is a contractively included subspace of $\HK\otimes\GG$.
 The following are equivalent:

(i) There exists a contractively contained subspace $\MM_G$ such that $\NN=\MM_G^\sharp$.

(ii) There exists a contraction $\tbB: \NN\to \BB\otimes\NN$, such that for any $\xi\in\NN$ are satisfied the relations
\begin{align}
\label{eq:gleason2}
 \xi&=(\bmB\otimes I_\GG)\tbB \xi+ (\pi\otimes I_\GG) \xi;\\
 \label{eq:inequality for difference quotients2}
 \|\tbB \xi\|^2_{\BB\otimes\NN}&\le \|\xi\|^2_{\NN}
-\|(\pi\otimes I_\GG)\xi\|^2_{\HK\otimes\GG}.
\end{align}

If $\NN$ is isometrically included in $\HK\otimes\GG$, then $\tbB=(\bmB^*\otimes I_\GG)|\NN$.
\end{theorem}

\begin{proof}
The implication (i)$\Rightarrow$(ii) follows from~\eqref{eq:gleason} and\eqref{eq:inequality for difference quotients} by taking $\tbB$ to be defined by~\eqref{eq:definition of tilde B}.

To prove (ii)$\Rightarrow$(i), let~$C:\NN\to \HK\otimes\GG$ be the inclusion (which is a contraction).
Note first that~\eqref{eq:inequality for difference quotients2} implies that the column matrix 
$\left(\begin{smallmatrix}
\tbB\\ (\tilde\pi\otimes I_\GG)C
\end{smallmatrix}\right)$
defines a contraction from $\NN$ to $(\BB\otimes \NN)\oplus \GG$, where we have denoted by $\tilde\pi$ the projection $\pi$ followed by the identification of its one dimensional range $\bbC\bm 1$ with the scalar field $\bbC$. This contraction may then be extended to a coisometry 
\begin{equation}\label{eq:definition of U for complementary spaces}
 U=\begin{pmatrix} \tbB & \bm b\\ (\tilde\pi\otimes I_\GG) & \bm d
 \end{pmatrix}: \NN\oplus \GG'\to (\BB\otimes \NN)\oplus \GG
\end{equation}
for some Hilbert space $\GG'$. Define then 
\[
 G(\lambda)= \bm d +(\tilde\pi\otimes I_\GG)C (I-Z_\NN(\lambda)\tbB )^{-1}Z_\NN(\lambda)\bm b.
\]

We claim that $\RR_C=\MM_G^\sharp$. Indeed, it follows from Corollary~\ref{co:M_Gsharp} (and the notations preceding it) that it is enough to show that $C=\gamma$, where $\gamma:\NN\to \HK\otimes\GG$ is defined by  
\[
 \gamma^*(k_\lambda\otimes y)= (I_\NN-\tbB^* Z_\NN(\lambda)^*)^{-1}C^*(\bm 1\otimes y).
\]

But  relation~\eqref{eq:gleason2} says that
\[
 C^*-\tbB^* (I_\BB\otimes C^*)(\bmB^*\otimes I_\GG)=C^*(\pi\otimes I_\GG),
\]
which, using
\[
\begin{split}
 (I_\BB\otimes C^*)(\bmB^*\otimes I_\GG)(k_\lambda\otimes y)&=
 (I_\BB\otimes C^*) (\beta(\lambda)\otimes k_\lambda\otimes y)=\beta(\lambda)\otimes C^*(k_\lambda\otimes y)\\
&= Z_\NN(\lambda)^*C^*(k_\lambda\otimes y),
\end{split}
\]
becomes
\[
 (I_\NN-\tbB^* Z_\NN(\lambda)^*) C^*(k_\lambda\otimes y)=C^*(\bm 1\otimes y),
\]
or
\[
 C^*(k_\lambda\otimes y)=(I_\NN-\tbB^* Z_\NN(\lambda)^*)^{-1}C^*(\bm 1\otimes y)=\gamma^*(k_\lambda\otimes y),
\]
which finishes the proof of the first part of the theorem.

Finally, suppose that $\NN$ is isometrically included into $\HK\otimes\GG$, so we may identify it to the image of $C$. The Hilbert space structure is then that of $\HK\otimes\GG$, and we have, by taking the scalar product with $\xi$ in~\eqref{eq:gleason2},
\begin{equation}\label{eq:gleason-scalar-product}
\|\xi\|^2=\<\tbB \xi, (\bmB^*\otimes I_\GG)\xi\>+\|(\pi\otimes I_\GG)\xi\|^2_{\HK\otimes\GG}.
\end{equation}
From~\eqref{eq:diff quotients for the whole space},~\eqref{eq:inequality for difference quotients2} and~\eqref{eq:gleason-scalar-product} it follows that
\[
\|(\bmB^*\otimes I_\GG)\xi\|_{\HK\otimes\GG}^2=\<\tbB \xi, (\bmB^*\otimes I_\GG)\xi\>\ge \|\tbB \xi\|^2_{\HK\otimes\GG}.
\]
Since~\eqref{eq:gleason-scalar-product} implies in particular that $\<\tbB \xi, (\bmB^*\otimes I_\GG)\xi\>$ is real, we have
\[
\|(\bmB^*\otimes I_\GG)\xi-\tbB \xi\|^2=\|(\bmB^*\otimes I_\GG)\xi\|^2-2\<\tbB \xi, (\bmB^*\otimes I_\GG)\xi\>+\|\tbB \xi\|^2\le0
\]
Thus $(\bmB^*\otimes I_\GG)\xi=\tbB \xi$ for any $\xi\in\NN$.
\end{proof}

Note that when $\NN$ is isometrically included in $\HK\otimes\GG$, we have directly proved that $(\bmB^*\otimes I_\GG)\NN\subset \BB\otimes\NN$. This also follows from the fact that in this case complementarity of subspaces becomes usual orthogonality, and so $\NN=\MM_G^\perp$; but we know that $(\bmB\otimes I_\GG)(\BB\otimes\MM_G)\subset \MM_G$ (see Theorem~\ref{th:invariant spaces} (ii)).

\subsection{The general situation} Let us consider now the general case of a Pick kernel (not necessarily normalized at $\lambda_0$). The normalization can be achieved by considering instead of $K(\mu, \lambda)$ the kernel $K'(\mu, \lambda)=\frac{K(\mu, \lambda)}{\delta(\mu)\delta(\bar\lambda)}$. The multipliers are the same for the two spaces $\HK$ and $\HHH(K')$.
The map $\Omega$ defined by $\Omega(f)=f/\delta$ is a unitary from $\HK$ to $\HHH(K')$. It commutes with multipliers and so maps the range of $G$ in $\HK$ onto the range of $G$ in $\HHH(K')$. 
Finally, from Lemma~\ref{le:identification of the complementary subspace by its kernel} it follows that it maps similarly the complements. The function $\beta$ and the operator $\bmB$ are the same for $\HHH(K)$ and $\HHH(K')$. One can then obtain the following analogue of Theorem~\ref{th:characterization of complementary subspaces}.

\begin{theorem}\label{th:complementary subspaces-general case}
 Suppose $K$ is a Pick kernel, $\GG$ is a Hilbert space, and $\NN$ is a contractively included subspace of $\HK\otimes\GG$. 
 The following are equivalent:

(i) There exists a contractively contained subspace $\MM_G$ such that $\NN=\MM_G^\sharp$.

(ii) There exists a contraction $\tbB: \Omega(\NN)\to \BB\otimes\Omega(\NN)$, such that for any $\xi\in\NN$ are satisfied the relations
\begin{align*}
%\label{eq:gleason2}
\Omega \xi&=(\bmB\otimes I_\GG)\tbB\Omega \xi+ (\pi'\otimes I_\GG)\Omega \xi;\\
% \label{eq:inequality for difference quotients2}
 \|\tbB \Omega\xi\|^2_{\BB\otimes\Omega\NN}&\le \|\Omega\xi\|^2_{\NN}
-\|(\pi'\otimes I_\GG)\Omega\xi\|^2_{\HHH(K')\otimes\GG},
\end{align*}
where $\pi'$ denotes the orthogonal projection onto the constant functions in $\HHH(K')$.
\end{theorem}

\section{Some examples}\label{se:final remarks}

Our main results, the characterization of ranges of multipliers and of their complementaries in a Pick space (Theorems~\ref{th:invariant spaces} and~\ref{th:characterization of complementary subspaces}) are exact analogues of the corresponding results in~\cite{ball-bolotnikov-fang2007} for the Drury--Arveson space $\DD(\bbB)$. Although this is not surprising by Theorem~\ref{th:embedding}, one should note that the results for a general Pick space do not follow from those for $\DD(\bbB)$; also, the multidimensional system theory used in~\cite{ball-bolotnikov-fang2007} cannot be transposed to the general case.
Moreover,  a series of facts that are true for the Drury--Arveson space  do not extend to a general Pick space. A few examples are given in this section.

First, the following uniqueness result concerning the representation of the function $\gamma$ is proved in~\cite{ball-bolotnikov-fang2007} for the space $\DD(\bbB)$. Note that in this case $\delta\equiv1$.

\begin{proposition}\label{pr:uniqueness a c}
Suppose that we have Hilbert spaces $\XX_1, \XX_2$ and operators $\gamma_i:\XX_i\to \HK\otimes \GG$ ($i=1,2$), such that $\ker \gamma_i=\{0\}$ and $\gamma_1\gamma_1^*=\gamma_2\gamma_2^*$. If, for any $x_i\in\XX_i$ ($i=1,2$) we have
\[
\gamma_i(x_i)=\bm c_i (I-Z_{\XX_i}(\lambda)\bm a_i)^{-1} x_i,
\]
then there exists a unitary operator $U:\XX_1\to\XX_2$ such that $\bm c_2=U\bm c_1$ and $(I_\BB\otimes U)\bm a_1=\bm a_2 U$.
\end{proposition}

This is no more true for a general Pick space, as shown by the  following very simple example.

\begin{example}\label{ex:not uniqueness}
Suppose $\Lambda=\{0, 1/2\}$, and the reproducing kernel $K$ is given by the formula
\[
K(\mu,\lambda)=\frac{1}{1-\frac{1}{2}(\mu\bar\lambda+\mu^2\bar\lambda^2)}
\]
This is a Pick kernel corresponding to the function $\beta:\Lambda\to\bbC^2$, $\beta(\lambda)=(\frac{ 
1} { \sqrt 2}\bar\lambda, \frac{ 1} { \sqrt 2}\bar\lambda^2)$.

Define $\XX_1=\XX_2=\GG=\bbC$, $\bm c_1=\bm c_2$ the identity on $\bbC$, and $\bm a_1,\bm a_2:\bbC\to\bbC^2$ given by
\[
\bm a_1(z)=(z/8,0),\quad \bm a_2(z)=(0, z/4).
\]
We have $Z_{\XX_i}(\lambda)(z,w)=\frac{ 1} { \sqrt 2}(\lambda z+\lambda^2 w)$ for $i=1,2$; thus $\gamma_i$ are elements of $\HK$  given by
\[
\gamma_1(\lambda)=\frac{ 1} { \sqrt 2}(1-\lambda/8)^{-1}, \quad \gamma_2(\lambda)=\frac{ 1} { \sqrt 2}(1-\lambda^2/4)^{-1}.
\]
Then $\gamma_1=\gamma_2$ on $\Lambda$, and therefore the hypothesis in Proposition~\ref{pr:uniqueness a c} is satisfied. However, it is obvious that there is no complex number $\kappa$ of modulus 1 such that $\kappa\bm a_1=\bm a_2\kappa$.
\end{example}

Another difference with the case of the Drury--Arveson space appears if we consider the possibility of taking $\tbB=\bmB^*\otimes I_\GG$. This is equivalent to the inclusion~\eqref{eq:inclusion with B}, since then~\eqref{eq:gleason for the whole space} and~\eqref{eq:diff quotients for the whole space} show that $\bmB^*\otimes I_\GG$ satisfies the conditions (ii) required from $\tbB$ in Theorem~\ref{th:characterization of complementary subspaces}. 

Suppose then that~\eqref{eq:inclusion with B} is satisfied. Then $\NN=\MM_G^\sharp$, whence the coisometry in~\eqref{eq:definition of U for complementary spaces} that defines $G$ has the form
\[
U=\begin{pmatrix}
\bmB^*\otimes I_\GG|\NN & \bm b\\
\tilde\pi \otimes I_\GG & \bm d
\end{pmatrix}.
\]
We have thus $\bm a= \bmB^*\otimes I_\GG|\NN$. 

There is more that can be said in this case about the operator $\bm a$.  We need some supplementary notation. Suppose $\XX$ is some Hilbert space, and define, for  each $\xi\in \BB$, the operator $L^\XX_\xi:\XX\to \BB\otimes\XX$ by $L^\XX_\xi(x)=\xi\otimes x$. It is easy to see that $(L^\XX_\xi)^*(\eta\otimes x)= \< \eta,\xi \>x$. If $\XX'$ is another Hilbert space and $A:\XX\to\XX'$, then
\[
 A(L^\XX_\xi)^*(\eta\otimes x)= \< \eta,\xi \>Ax= (L^{\XX'}_\xi)^*(I_\BB\otimes A)(\eta\otimes x)
\]
and thus
\begin{equation}\label{eq:commutation of Lxi}
  A(L^\XX_\xi)^*=(L^{\XX'}_\xi)^*(I_\BB\otimes A).
\end{equation}

Define then $\bm a_\xi:\XX\to\XX$ by $\bm a_\xi= (L^\XX_\xi)^*\bm a$. The operators $\bm a_\xi$ can be considered as ``coordinates'' of the operator $\bm a$. 
We may then prove the following commutativity result.

\begin{lemma}\label{le:if N is invariant, then a_xi commute}
 With the above notations, if we take $\bm a= \bmB^*\otimes I_\GG|\NN$, then $ \bm a_\xi \bm a_\eta=\bm a_\eta \bm a_\xi$ for all $\xi, \eta\in\BB$.
\end{lemma}

\begin{proof}
 For $\xi\in\BB$, denote $\ell_\xi:= L^{\HK\otimes\GG}$. We have
\[
 \ell_\xi^* (\bmB^*\otimes I_\GG)(k_\lambda\otimes y)
=\ell_\xi^* (\beta(\lambda)\otimes k_\lambda\otimes y)=
\<\beta(\lambda),\xi\> (k_\lambda\otimes y).
\]
So, for any $\xi\in\BB$ the operator $\ell_\xi^* (\bmB^*\otimes I_\GG)$ has the total set $\{ k_\lambda\otimes y:\lambda\in\Lambda, y\in\GG\}$ as eigenvectors. It follows that they all commute (for different values of $\xi\in\BB$).

Let us then consider the diagram 
\[
 \begin{CD}
\NN @> \bmB^*\otimes I_\GG|\NN>> \BB\otimes\NN @> (L^\NN_\xi)^* >> \NN\\
@V \iota_\NN VV @V I_\BB\otimes\iota_\NN VV @V \iota_\NN VV\\
\HK\otimes\GG  @> \bmB^*\otimes I_\GG>> \BB\otimes\HK\otimes\GG @> \ell_\xi^* >> \HK\otimes\GG
 \end{CD}
\]
The first square is commutative by the remark above, while the second is commutative by applying~\eqref{eq:commutation of Lxi} to the case $\XX=\NN$, $\XX'=\HK\otimes\GG $, $A=\iota_\NN$. It follows that
$\iota_\NN \bm a_\xi= \ell_\xi^* (\bmB^*\otimes I_\GG) \iota_\NN$, whence we deduce that all the operators $\bm a_\xi$ (for $\xi\in\BB$) commute.
\end{proof}

In~\cite{ball-bolotnikov-fang2007} one obtains a kind of converse to Lemma~\ref{le:if N is invariant, then a_xi commute} in the case of the Drury--Arveson space. Namely, the next proposition follows from~\cite[Theorem 3.15]{ball-bolotnikov-fang2007}.

\begin{proposition}\label{pr:abelian case in [BBF]}
Suppose $\HK=\DD(\bbB)$. If $\bm c \bm a_\xi \bm a_\eta=\bm c \bm a_\eta \bm a_\xi$ for all $\xi,\eta\in \BB$, then
\begin{equation}\label{eq:formula for tilde B=B*}
(\bmB^*\otimes I_\GG) \gamma= (I_\BB\otimes\gamma)\bm a.
\end{equation}
Consequently, $(\bmB^*\otimes I_\GG)(\NN)\subset \BB\otimes \NN$ and $\tbB=(\bmB^*\otimes I_\GG)|\NN$.
\end{proposition}

As shown by the next result, the validity of this proposition does not extend to a general Pick space. We assume again the normalization condition in the next theorem.

\begin{theorem}\label{th:true only for drury arveson} Suppose that $K$ is normalized at $\lambda_0$, and the image of $\bar\beta$ is not a set of uniqueness for $\DD(\bbB)$. Then there exists a unitary $U= \left(
\begin{smallmatrix}
\bm a&\bm b\\ \bm c& \bm d
\end{smallmatrix}
\right):\bbC\oplus \BB\to \BB\oplus \bbC$ such that, if $G$ is defined by~\eqref{eq:representation of G}, then $\bm a_\eta \bm a_{\eta'}= \bm a_{\eta'} \bm a_\eta$ for all $\eta,\eta'\in \BB$, but $\bmB^*(\MM_G^\sharp)\not\subset \BB\otimes \MM_G^\sharp$. 
\end{theorem}

\begin{proof} As seen in the statement, we intend to define $\XX=\GG=\bbC$.
Since the image of $\bar\beta$ is not a set of uniqueness for $\DD(\bbB)$, the closed linear span of $\ddd_\chi$ with $\chi$ in the image of $\bar\beta$ is not the whole $\DD(\bbB)$.  Take a vector $\xi\in\BB$, $\xi\not=0$, such that $\ddd_{J\xi}$ is not in this span. Define 
\[
 \bm a(1)=\xi, \quad \bm c(1)=\sqrt{1-\|\xi\|^2}
\]
and complete $U$ to be a unitary as required.

 The commutativity relation is obviously satisfied, since the operators $a_\eta$ act on the space $\XX$ of dimension~1. Also, since $\dim \XX=1$, we have $Z_\XX(\lambda)=B(\lambda)$, and $\gamma$ is defined by $f_0=\gamma(1)$, which is a function in $\HK$, namely 
\[
f_0(\lambda)=\sqrt{1-\|\xi\|^2} \frac{1}{1-B(\lambda)\bm a}=
 \frac{\sqrt{1-\|\xi\|^2}}{1-\< \xi, \beta(\lambda)\> }.
\]
The space $\MM_G^\sharp$ is the one-dimensional space spanned by~$f_0$.

Suppose that $\bmB^*(\MM_G^\sharp)\subset \BB\otimes \MM_G^\sharp$; this would mean that for some vector $\eta\in \BB$ we have $\bmB^* f_0=\eta\otimes f_0$. Then, for any $\lambda\in\Lambda$,
\[
\begin{split}
\< \eta,\beta(\lambda) \> f_0(\lambda)&
= \< \eta\otimes f_0,\beta(\lambda)\otimes k_\lambda \>
=\< \bmB^* f_0,\beta(\lambda)\otimes k_\lambda \>
= \< f_0,\bmB (\beta(\lambda)\otimes k_\lambda) \>\\
&= \< f_0,\bmB \bmB^* k_\lambda \>
= \< f_0,k_\lambda-\pi k_\lambda \>
=f_0(\lambda)- f_0(\lambda_0),
\end{split}
\]
and therefore 
\[
f_0(\lambda) = \frac{f_0(\lambda_0)}{1-\< \eta, \beta(\lambda)\> }
=\frac{\sqrt{1-\|\xi\|^2}}{1-\< \eta, \beta(\lambda)\> },
\]
where we have used $\beta(\lambda_0)=0$ (as a consequence of the normalization assumption).
Since the family $\beta(\lambda)$, $\lambda\in\Lambda$ is total, it follows that $\eta=\xi$.

For any $\zeta\in \BB$, denote as above $\ell_\zeta:= L_\zeta^{\HK\otimes\GG}$.
Then
\[
\ell_\zeta^*\bmB^* f_0=\ell_\zeta^* (\xi\otimes f_0)= \< \xi,\zeta \>f_0.
\]
and thus $f_0$ is an eigenvector for all operators $\ell_\zeta^* \bmB^*$, $\zeta\in \BB$, of eigenvalue $\< \xi,\zeta \>$. On the other hand, if $\lambda\in\Lambda$, then
\[
\ell_\zeta^*\bmB^*k_\lambda=\ell_\zeta^* (\beta(\lambda)\otimes k_\lambda) = \< \beta(\lambda),\zeta \> k_\lambda =
\overline{(\psi_{J\zeta}\circ \bar\beta)(\lambda)}k_\lambda,
\]
(remember that $\psi_{J\zeta}(\eta)=\<\eta, J\zeta\>$).

By Lemma~\ref{le:M*phi J=J M*phi o b}, it follows that $\epsilon_Kf_0$ is an eigenvector for all adjoints of the multipliers  $p_{J\zeta}$ (defined at the end of Section~\ref{se:pick spaces}),  and therefore $\epsilon_Kf_0$ is a multiple of $\ddd_\chi$ for some $\chi\in\BB$.  Since $f_0$ is a multiple of $\epsilon_K^*\ddd_{J\xi}$, we have $\epsilon_K\epsilon_K^*\ddd_{J\xi}=\alpha\ddd_\chi$  for some $\alpha\in\bbC$, whence $\epsilon_K^*\ddd_{J\xi}=\epsilon_K^*\alpha\ddd_\chi$, or $\ddd_{J\xi}\circ\bar\beta=\alpha\ddd_\chi\circ\bar\beta$. So
\[
\frac{1}{1- \< \bar\beta(\lambda),J\xi \>}= \frac{\alpha}{1- \< \bar\beta(\lambda),\chi \>}
\]
for all $\lambda\in\Lambda$. Taking $\lambda=\lambda_0$ yields $\alpha=1$; then using again the fact that the family $\bar\beta(\lambda)$, $\lambda\in\Lambda$, is total, it follows that $\chi=J\xi$.

Therefore $\ddd_{J\xi}=\epsilon_Kf_0$ belongs to the range of $\epsilon_K$. This contradicts the assumption, since this range is the linear span of $\ddd_\chi$ with $\chi$ in the image of $\bar\beta$. Therefore $\bmB^*(\MM_G^\sharp)\not\subset \BB\otimes \MM_G^\sharp$.
\end{proof}

%\bibliographystyle{plain}	% (uses file "plain.bst")

%\bibliography{myrefs}

%
%
%
%
%
%
%

\end{document}